\documentclass[12pt, a4paper, reqno]{amsart}

\usepackage[english]{babel}

\usepackage[T1]{fontenc}
\usepackage{lmodern}

\usepackage{comment}

\usepackage{amsmath}
\usepackage{amssymb}
\usepackage{amsthm}

\usepackage{url}

\usepackage{tikz-cd}

\usepackage{hyperref} 

\usepackage{enumitem}
\usepackage[normalem]{ulem}

\newcommand{\lexp}[2]{\null^{#2} \mkern-2mu #1}
\newcommand{\lexpp}[2]{\null^{#2} \mkern-2mu (#1)}
\newcommand{\lexpq}[2]{(\null^{#2} \mkern-2mu #1)}

\renewcommand{\epsilon}{\varepsilon}
\renewcommand{\theta}[0]{\vartheta}
\renewcommand{\phi}[0]{\varphi}

\newcommand{\trivial}[0]{\mathrm{triv}}

\newcommand{\lift}[1]{#1^{\uparrow}}
\newcommand{\wasoverline}[1]{#1}

\newcommand{\N}{\mathbb{N}}
\newcommand{\Z}{\mathbb{Z}}

\newcommand{\Size}[1]{\left\lvert #1 \right\rvert}

\newcommand{\Span}[1]{\left\langle\, #1 \,\right\rangle}

\newcommand{\Set}[1]{\left\{ #1 \right\}}

\newcommand{\Zero}[0]{\mathbf{0}}

\DeclareMathOperator{\GL}{GL}

\DeclareMathOperator{\End}{End}
\DeclareMathOperator{\Aut}{Aut}

\DeclareMathOperator{\Frat}{Frat}

\DeclareMathOperator{\Perm}{Perm}

\newtheorem{dummy}{Dummy}
\numberwithin{dummy}{section}
\numberwithin{figure}{section}

\newtheorem{theorem}[dummy]{Theorem}

\newtheorem{lemma}[dummy]{Lemma}

\newtheorem{proposition}[dummy]{Proposition}

\newtheorem{corollary}[dummy]{Corollary} 

\theoremstyle{definition}
\newtheorem{definition}[dummy]{Definition}

\newtheorem{example}[dummy]{Example}

\theoremstyle{remark}

\newtheorem{remark}[dummy]{Remark}

\makeatletter
\def\imod#1{\allowbreak\mkern10mu({\operator@font mod}\,\,#1)}
\makeatother

\numberwithin{equation}{section}

\usepackage{amsthm, mathtools, amssymb, amsfonts, stmaryrd, latexsym,
  mathrsfs, todonotes, color}

\newcommand{\kN}{\mathfrak{N}}
\DeclareMathOperator{\id}{id}

\begin{document}

\date{25 July 2022, 10:11 CEST --- Version 4.03}

\title[Brace blocks]
    {Brace blocks from bilinear maps\\
      and liftings of endomorphisms}    
\author{A. Caranti}

\address[A.~Caranti]%
 {Dipartimento di Matematica\\
  Universit\`a degli Studi di Trento\\
  via Sommarive 14\\
  I-38123 Trento\\
  Italy} 

\email{andrea.caranti@unitn.it} 

\urladdr{https://caranti.maths.unitn.it/}

\author{L. Stefanello}

\address[L.~Stefanello]
        {Dipartimento di Matematica\\
          Universit\`a di Pisa\\
          Largo Bruno Pontecorvo, 5\\
          I-56127 Pisa\\
          Italy}
\email{lorenzo.stefanello@phd.unipi.it}

\urladdr{https://people.dm.unipi.it/stefanello}


\thanks{Both authors are members of INdAM---GNSAGA. The first author
gratefully acknowledges support from the Department of Mathematics of
the University of Trento. The second author
gratefully acknowledges support from the Department of Mathematics of
the University of Pisa.}

\begin{abstract}
  We  extend two  constructions of  Alan Koch,  exhibiting methods  to
  construct brace blocks,  that is, families of group  operations on a
  set $G$ such that any two of them induce a  skew brace structure on
  $G$.

  We construct these operations by using bilinear maps and liftings
  of endomorphisms of quotient groups with respect to a central
  subgroup.

  We provide several examples of  the construction, showing that there
  are brace blocks  which consist of distinct operations  of any given
  cardinality.

  One of the examples we give yields an answer to a question of Cornelius
  Greither. This example exhibits a sequence of distinct operations on the
  $p$-adic Heisenberg group $(G, \cdot)$ such that any two operations
  give a skew brace structure on $G$ and the sequence of operations
  converges to the original operation ``$\cdot$''.
\end{abstract}

\maketitle

\thispagestyle{empty}

\section{Introduction}

Alan   Koch    gave   in~\cite{Koch-Abelian}    the   following
construction.  Let $(G,  \cdot)$  be a  group, and  let  $\psi$ be  an
endomorphism of $G$ with an abelian image. Define an operation on $G$ via
\begin{equation}
  \label{eq:K1}
  g \circ h
  =
  g \cdot \lexpq{g}{\psi}^{-1} \cdot h \cdot \lexp{g}{\psi},
\end{equation}
where we  denote by  $\lexp{g}{\psi}$ the  value of  the map  $\psi$ on
$g$. Koch proved that $(G, \cdot, \circ)$ is a bi-skew brace.

In~\cite{CS1} we
characterised the endomorphisms $\psi$ of $G$ such that the operations
\begin{equation}
  \label{eq:pmepsilon}
  g \circ h
  =
  g \cdot \lexpq{g}{\psi}^{\epsilon} \cdot h \cdot \lexpq{g}{\psi}^{-\epsilon}
\end{equation}
yield (bi-)skew braces $(G, \cdot, \circ)$, for $\epsilon = \pm 1$. In
particular, when
\begin{equation}
  \label{eq:psi-psi}
  [ \lexp{G}{\psi}, \lexp{G}{\psi} ] \le Z(G),
\end{equation}
all
operations~\eqref{eq:pmepsilon} yield bi-skew braces $(G, \cdot, \circ)$,
for $\epsilon = \pm 1$.

In~\cite{Koc22} Koch  showed that  if $(G, \cdot)$  is a  group and
$\psi$ is an endomorphism of $G$ with an abelian image, then $\psi$ is
also  an  endomorphism of  the  group  $(G,  \circ)$, for  $\circ$  as
in~\eqref{eq:K1}. Iterating this procedure, Koch was able to construct
a countable family $( \circ_{n} : n \in
\N)$ of  group operations on the  set $G$, starting with  the original
operation  ``$\cdot$''  as  $\circ_{0}$,  such  that  $(G,  \circ_{n},
\circ_{m})$ is a skew brace for all $n, m \ge 0$. Koch named such a
family a \emph{brace block}.

In this  paper we  extend both of  Koch's constructions.  An immediate
generalisation  could be  obtained  by replacing  the assumption  that
$\psi$ has  an abelian  image with  assumption~\eqref{eq:psi-psi}.  We
work, however,  in the  following slightly  more general  context. Let
$(G, \cdot)$  be a group, and  let $K \le  A \le G$ be  subgroups such
that $K \le Z(G)$ and $A/K$ is abelian. Consider the set
\begin{equation*}
  \mathcal{A}
  =
  \Set{
    \psi \in \End(G/K)
    :
    \lexpp{G/K}{\psi}\le A/K
  },
\end{equation*}
which is a ring under the usual operations. Given $\psi \in \mathcal{A}$,
let $\lift{\psi}$ be a lifting of 
${\psi}$ to $G$, that is, a map on $G$ such that
\begin{equation*}
  (\lexp{g}{\lift{\psi}}) K
  =
  \lexpp{gK}{{\psi}}
\end{equation*}
for $g \in  G$.  Such  a  $\lift{\psi}$ need  not be  unique, nor  an
endomorphism of $G$.

We could  try and  extend Koch's iterative  construction by  using the
$\psi$  and $\lift{\psi}$,  possibly employing  different elements  of
$\mathcal{A}$ at each step. However, in doing so certain bilinear maps
occur  naturally,  in   the  form  of  products   of  suitable  central
commutators.   We   thus  incorporate  these  bilinear   maps  in  our
construction,  by   introducing,  on  the  model   of  a  construction
in~\cite{p4}, the set
\begin{equation*}
  \mathcal{B}
  =
  \Set{
    \alpha\colon G\times G\to K
    :
    \text{$\alpha$ is bilinear, and $\alpha(K,G) = \alpha(G,K) = 1$}
  }.
\end{equation*}

Given $(\psi, \alpha) \in \mathcal{A} \times \mathcal{B}$, we can now
define on $G$ the operation 
\begin{equation*}
  g\circ_{\psi,\alpha} h
  =
  g
  \cdot
  \lexp{g}{\lift{\psi}} \cdot h \cdot \lexpq{g}{\lift{\psi}}^{-1}
  \cdot
  \alpha(g,h),
\end{equation*}
which does not depend on the  choice of the particular
lifting $\lift{\psi}$ of $\psi$.

In Theorem~\ref{thm:maintheorem} we show that 
\begin{equation*}
  (G, \circ_{\psi, \alpha}, \circ_{\phi, \beta})
\end{equation*}
is  a  skew brace  for  all  $({\psi},  \alpha), ({\phi},  \beta)  \in
\mathcal{A}  \times  \mathcal{B}$.   In other  words,  the  operations
``$\circ_{\psi, \alpha}$''  form a brace block,  for $({\psi}, \alpha)
\in    \mathcal{A}   \times    \mathcal{B}$.    Koch's    construction
in~\cite{Koch-Abelian} occurs  here by specialising  $K = 1$  (so that
both $\alpha$ and $\beta$ are the trivial bilinear map), replacing $\phi$ with
$g \mapsto \lexp{g}{\phi}^{-1}$, and  taking $\psi = \Zero$ to be
the endomorphism of $G$ mapping every element $g \in G$ to $1$. 

In  Theorem~\ref{thm:main2}  we  also  show  that,  given  a  sequence
$(\psi_{n},  \alpha_{n})  \in  \mathcal{A}  \times  \mathcal{B}$,  the
operations ``$\circ_{\psi_{n}, \alpha_{n}}$'' can be iterated,
yielding  operations  of  the  same form,
thereby extending Koch's construction of brace blocks in~\cite{Koc22}.
As  mentioned above,  it  is in  this step  that  our introduction  of
bilinear  maps proves  to be  necessary, as such
maps occur  naturally in  the iterative steps  as products  of central
commutators    (see    Theorem~\ref{thm:main2}~and    the    following
Remark~\ref{rem:the-reason-why}).

Perhaps it could be noted that, given $(\psi, \alpha), (\phi, \beta)
\in \mathcal{A} \times \mathcal{B}$, when doing our construction
first with $(\psi, \alpha)$ and then with
$(\phi, \beta)$, we obtain an operation $\circ_{(\xi, \gamma)}$, for a
certain $(\xi, \gamma) \in \mathcal{A} \times \mathcal{B}$, where
$\xi = \psi + \phi + \psi \phi$ is the Jacobson circle of $\psi$
and $\phi$ (see Theorem~\ref{thm:main2}). 

We  provide  several  examples  of   our  construction  at  work.

In
particular,  Example~\ref{ex:CG} shows that we can obtain brace blocks 
of any given cardinality consisting of distinct operations, and 
answers  a  question  posed to  us  by
Cornelius  Greither at  the  Conference on  Hopf  algebras and  Galois
module theory held (online) in Omaha  in May 2021.  Namely,
we exhibit  a sequence
of distinct (topological) group  operations on the $p$-adic Heisenberg
group $(G, \cdot)$ such that $G$ is  a skew brace with respect to any
two  of them,  and the  sequence converges  to the  original operation
``$\cdot$''.

The  interplay between  set-theoretic  solutions  of the  Yang--Baxter
equation of Mathematical Physics,  (skew) braces, regular subgroups of
the holomorph, and Hopf--Galois  structures has spawned a considerable
body of  literature in  recent years (see,  for example,~\cite{braces,
  skew, CJO, GP, Byo96, SV2018}).

We  discuss  some  preliminaries   about  skew  braces  in
  Section~\ref{sec:preliminaries}.    In  Section~\ref{sec:main}   we
describe    our    construction,    which    we    then    apply    in
Section~\ref{sec:iterating}  to  the  extension  of  Koch's  iterative
approach   mentioned   above.    Section~\ref{sec:examples}   contains
examples.   We discuss  set-theoretic  solutions  of the  Yang--Baxter
equation  in  Section~\ref{sec:YB}.   In  Section~\ref{sec:graphs}  we
rephrase  brace  blocks  in  terms of  the \emph{normalising  graph},
exploiting the  connections between regular  subgroups and
  skew braces.

We are very grateful to Cornelius Greither for the suggestion that led to
Example~\ref{ex:CG}, and to the referee for a number of useful suggestions.

\section{Preliminaries}
\label{sec:preliminaries}

In what follows, if $(G,\cdot)$ is  a group, we write $\End(G, \cdot)$
for the monoid of endomorphisms of  $(G, \cdot)$. We denote the action
of  $\psi  \in  \End(G, \cdot)$  on  $g  \in  G$  by a  left  exponent
$\lexp{g}{\psi}$,   and    thus   we   compose    such   endomorphisms
right-to-left.    If    $\epsilon\in\Z$,   then   we    have   clearly
$\lexp{(g^{\epsilon})}{\psi}  = (\lexp{g}{\psi})^{\epsilon}$;  thus we
write     simply    $\lexp{g}{\psi}^{\epsilon}$     for    such     an
element. 

We recall here the definitions of skew braces, bi-skew braces, and brace blocks, 
given in~\cite{skew},~\cite{Childs-bi-skew}, and~\cite{Koc22} respectively. (We slightly generalise the last definition, in order to possibly admit brace blocks with
arbitrary cardinality.)

Let $(G,\cdot)$  be a group with  identity $1$. We denote  by $g^{-1}$
the inverse of an element $g\in G$ with respect to ``$\cdot$''.

\begin{definition}
  A  \emph{skew (left)  brace}  is a  triple $(G,\cdot,\circ)$,  where
  $(G,\cdot)$ and $(G,\circ)$ are groups, and for all $g,h,k\in G$,
  \begin{equation*}
    g\circ(h\cdot k) = (g\circ h)\cdot g^{-1}\cdot (g\circ k).
  \end{equation*}
\end{definition}
It is immediate to check that $(G,\cdot)$ and
$(G,\circ)$  share  the   same  identity  $1$; see~\cite[Lemma 1.7]{skew}. 

\begin{definition}
  A \emph{bi-skew  brace} is  a triple  $(G,\cdot,\circ)$, where
  both $(G,\cdot,\circ)$ and $(G,\circ,\cdot)$ are skew braces.
\end{definition}

\begin{definition}
  Let $G$ be a non-empty set. A \emph{brace block} on $G$ is a family
  $\mathcal{F}$ of group operations on $G$ such that $(G, \circ,
  \diamond)$ is a (bi-)skew brace for all $\circ,
  \diamond \in \mathcal{F}$.
\end{definition}

Also, we  write $\iota$  for the homomorphism  mapping $g\in  G$ to
conjugation by $g$:
\begin{align*}
	\iota\colon (G,\cdot)&\to \Aut(G,\cdot)\\
	g&\mapsto (x\mapsto g\cdot x\cdot g^{-1}).
\end{align*} 

\section{The main construction}
\label{sec:main}

Let $(G,\cdot)$ be a group, let $K$  be a subgroup of $G$ contained in
the centre $Z(G)$ of  $G$, and let $A$ be a subgroup  of $G$ such that
$A/K$ is abelian.

Define
\begin{equation*}
  \mathcal{A}
  =
  \Set{
    \wasoverline{\psi}\in\End(G/K)
    :
    \lexpp{G/K}{\wasoverline{\psi}}\le A/K
  }.
\end{equation*}
The set $\mathcal{A}$  is a  ring under the usual  operations, with
unity only when $A = G$.

Given $\wasoverline{\psi}\in \mathcal{A}$,  define $\lift{\psi}$ to be
a  lifting  of  $\wasoverline{\psi}$,  that is,  a  set-theoretic  map
$\lift{\psi}  \colon G  \to  A$  such that  the  following diagram  is
commutative:
\begin{equation*}
  \begin{tikzcd}
    G \arrow[r, "\lift{\psi}"] \arrow[d] & A \arrow[d] \\
    G/K \arrow[r, "\wasoverline{\psi}"']          & A/K          
  \end{tikzcd}
\end{equation*}
This is equivalent to saying that for all $g\in G$ we have
\begin{equation*}
  (\lexp{g}{\lift{\psi}}) K
  =
  \lexpp{g K}{\wasoverline{\psi}}.
\end{equation*}
Liftings are clearly not unique if $K \ne 1$, and need not be
endomorphisms of $G$.

Given a lifting $\lift{\psi}$ of $\psi \in \mathcal{A}$, we have for  $g,h\in G$ that
\begin{align*}
  \lexpp{g\cdot h}{\lift{\psi}}K
  &=
  \lexpp{(g\cdot h)K}{\wasoverline{\psi}}
  =
  \lexpp{g K}{\wasoverline{\psi}}\cdot \lexpp{h K}{\wasoverline{\psi}}
  \\&=
  (\lexp{g}{\lift{\psi}}) K\cdot (\lexp{h}{\lift{\psi}}) K
  =
  (\lexp{g}{\lift{\psi}}\cdot\lexp{h}{\lift{\psi}})K,
\end{align*}
that is,
\begin{equation*}
  \lexpp{g\cdot h}{\lift{\psi}}
  \equiv
  \lexp{g}{\lift{\psi}}\cdot\lexp{h}{\lift{\psi}}
  \pmod{K}.
\end{equation*}
Similarly, as $A/K$ is abelian, we obtain
\begin{align*}
  \lexp{g}{\lift{\psi}}\cdot\lexp{h}{\lift{\psi}}
  \equiv
  \lexp{h}{\lift{\psi}}\cdot\lexp{g}{\lift{\psi}}
  \pmod{K};
\end{align*}
since $K\le Z(G,\cdot)$, this yields
\begin{equation}
  \label{eq:iota}
  \iota(\lexpp{g\cdot h}{\lift{\psi}})
  =
  \iota(\lexp{g}{\lift{\psi}}\cdot\lexp{h}{\lift{\psi}})
  =
  \iota(\lexp{h}{\lift{\psi}}\cdot\lexp{g}{\lift{\psi}}).
\end{equation}

\begin{remark}
  Let    $\wasoverline{\psi},\wasoverline{\phi}\in\mathcal{A}$    have
  liftings $\lift{\psi}$  and  $\lift{\phi}$, respectively.  Then the
  map $\lift{\psi} + \lift{\phi} \colon G \to A$ defined by
  \begin{equation*}
    \lexp{g}{\lift{\psi}+\lift{\phi}}
    =
    \lexp{g}{\lift{\psi}}\cdot \lexp{g}{\lift{\phi}}
  \end{equation*}
  is   a  lifting   of  $\wasoverline{\psi}+\wasoverline\phi$,   and  the   map
  $\lift{\psi}\lift{\phi}\colon G\to A$ defined by
  \begin{equation*}
    \lexp{g}{\lift{\psi}\lift{\phi}}
    =
    \lexpp{\lexp{g}{\lift{\phi}}}{{\lift{\psi}}}
  \end{equation*}
  is a lifting of $\wasoverline{\psi}\wasoverline\phi$.
\end{remark}

Now consider 
\begin{equation*}
  \mathcal{B}
  =
  \Set{
    \alpha\colon G\times G\to K
    :
    \text{$\alpha$ is bilinear, and $\alpha(K,G)=\alpha(G,K)=1$}
  }.
\end{equation*}
(Similar maps were used in a related  context in~\cite{p4}.)  If
$(\wasoverline\psi,\alpha)\in\mathcal{A}\times    \mathcal{B}$    and
$\lift{\psi}$  is  a  lifting  of  $\wasoverline\psi$, then we  define  an
operation on $G$ via
\begin{equation*}
  g\circ_{\psi,\alpha} h
  =
  g\cdot \lexp{g}{\lift{\psi}}\cdot h\cdot
  (\lexp{g}{\lift{\psi}})^{-1}\cdot \alpha(g,h) 
  =
  g\cdot \lexp{h}{\iota(\lexp{g}{\lift{\psi}})}\cdot \alpha(g,h).
\end{equation*}
Note   that   the   operation  ``$\circ_{\psi,\alpha}$''   is   indeed
independent of the  particular choice of the  lifting $\lift{\psi}$ of
$\psi$, as if $\psi^{\uparrow\uparrow}$  is another lifting of $\psi$,
we            have            $\lexp{g}{\lift{\psi}}            \equiv
\lexp{g}{\psi^{\uparrow\uparrow}}  \pmod{K}$ for  $g \in  G$, so  that
$\iota(\lexp{g}{\lift{\psi}})                                        =
\iota(\lexp{g}{\psi^{\uparrow\uparrow}})$, as $K \le Z(G,\cdot)$.

We claim  that $(G,\cdot,\circ_{\psi,\alpha})$ is a  bi-skew brace.  A
sharper statement actually holds.
\begin{theorem}\label{thm:maintheorem}
  For $(\phi,\beta), (\psi,\alpha) \in \mathcal{A} \times \mathcal{B}$
  we have that
  \begin{equation*}
    (G,\circ_{\phi,\beta},\circ_{\psi,\alpha})
  \end{equation*}
  is a bi-skew brace.

  In other words, the family $(\circ_{\phi, \beta} : (\phi, \beta)  \in
  \mathcal{A} \times \mathcal{B})$ is a brace block.
\end{theorem}

Before going into the proof, we record the following  useful fact,
which is  a slight generalisation of~\eqref{eq:iota}:
\begin{equation*}
  \iota(\lexp{h}{\lift{\phi}}\cdot \lexp{g}{\lift{\psi}})
  =
  \iota(\lexp{g}{\lift{\psi}}\cdot \lexp{h}{\lift{\phi}}).
\end{equation*}
In fact we have
\begin{equation*}
  (\lexp{h}{\lift{\phi}}\cdot \lexp{g}{\lift{\psi}})K
  =
  (\lexp{h}{\lift{\phi}}) K\cdot (\lexp{g}{\lift{\psi}}) K
  =
  (\lexp{g}{\lift{\psi}}) K\cdot  (\lexp{h}{\lift{\phi}}) K
  =
  (\lexp{g}{\lift{\psi}}\cdot \lexp{h}{\lift{\phi}})K,
\end{equation*}
as $A / K$ is abelian.

\begin{proof}
  Let $\lift{\phi}$ and $\lift{\psi}$ be liftings of $\phi$ and
  $\psi$, respectively.
  
It is not difficult to verify that $(G,\circ_{\phi,\beta})$ is indeed
a group. Associativity follows from the fact that
\begin{equation*}
  (g \circ_{\phi, \beta} h) \circ_{\phi, \beta} k
  \quad\text{and}\quad
  g \circ_{\phi, \beta} (h \circ_{\phi, \beta} k)
\end{equation*}
both expand to
\begin{equation*}
  g \cdot \lexp{g}{\lift{\phi}} \cdot
  h \cdot \lexp{h}{\lift{\phi}} \cdot
  k \cdot \lexpq{h}{\lift{\phi}}^{-1} \cdot
  \lexpq{g}{\lift{\phi}}^{-1} \cdot
  \beta(g, h) \cdot \beta(g, k) \cdot \beta(h, k).
\end{equation*}

The identity is $1$ and the inverse of $g$ is
\begin{equation*}
  \overline{g}
  =
  \lexpq{g}{\lift\phi}^{-1}
  \cdot
  g^{-1}\cdot \lexp{g}{\lift\phi}
  \cdot
  \beta(g,g).
\end{equation*}

To show that $(G,\circ_{\phi,\beta},\circ_{\psi,\alpha})$ is a skew
brace,  we need to verify that  for all $g,h,k\in G$ we have
  \begin{equation*}
    (g\circ_{\psi,\alpha} h)
    \circ_{\phi,\beta}
    \overline{g}
    \circ_{\phi,\beta}
    (g\circ_{\psi,\alpha} k)
  \end{equation*}
  equals 
  \begin{equation*}
    g
    \circ_{\psi,\alpha}
    (h\circ_{\phi,\beta}k).
  \end{equation*}
  We have
  \begin{align*}
    &\overline{g}\circ_{\phi,\beta}(g\circ_{\psi,\alpha} k)=\overline{g}\circ_{\phi,\beta}(g\cdot
    \lexp{k}{\iota(\lexp{g}{\lift{\psi}})}\cdot
    \alpha(g,k))\\
    &\qquad= (\lexp{g}{\lift{\phi}})^{-1}\cdot g^{-1}\cdot
    \lexp{g}{\lift{\phi}}\cdot \beta(g,g)\cdot \lexpp{g\cdot
      \lexp{k}{\iota(\lexp{g}{\lift{\psi}})}\cdot
      \alpha(g,k)}{\iota(\lexpp{g^{-1}}{\lift{\phi}})}\\&\qquad \quad\cdot
    \beta(g^{-1},g)\cdot \beta(g^{-1},k)
    \\
    &\qquad=
    (\lexp{g}{\lift{\phi}})^{-1}\cdot g^{-1} \cdot
    \lexpp{g\cdot                 \lexp{k}{\iota(\lexp{g}{\lift{\psi}})}}{\iota(\lexp{g}{\lift{\phi}}\cdot\lexpp{g^{-1}}{\lift{\phi}})}\cdot \lexp{g}{\lift{\phi}}\cdot \beta(g^{-1},k)\cdot
    \alpha(g,k)\\
    &\qquad =(\lexp{g}{\lift{\phi}})^{-1}\cdot   \lexp{k}{\iota(\lexp{g}{\lift{\psi}})}\cdot \lexp{g}{\lift{\phi}}\cdot \beta(g^{-1},k)\cdot
    \alpha(g,k)
  \end{align*}
  and 
  \begin{align*}
    &(g\circ_{\psi,\alpha} h)\circ_{\phi,\beta}\overline{g}\circ_{\phi,\beta} (g\circ_{\psi,\alpha} k)\\
    &\qquad= g\cdot \lexp{g}{\lift{\psi}}\cdot h\cdot(\lexp{g}{\lift{\psi}})^{-1}\cdot \alpha(g,h)\\
    &\qquad\quad\cdot \lexpp{(\lexp{g}{\lift{\phi}})^{-1}\cdot   \lexp{g}{\lift{\psi}}\cdot k\cdot(\lexp{g}{\lift{\psi}})^{-1}            \cdot \lexp{g}{\lift{\phi}}\cdot \beta(g^{-1},k)\cdot
      \alpha(g,k)}{\iota(\lexpp{g\cdot h}{\lift{\phi}})}\\
    &\qquad\quad\cdot \beta(g, k)\cdot \beta(h, k)\\
    &\qquad= g\cdot \lexp{g}{\lift{\psi}}\cdot h\cdot(\lexp{g}{\lift{\psi}})^{-1}\cdot\lexpp{(\lexp{g}{\lift{\phi}})^{-1}\cdot   \lexp{g}{\lift{\psi}}\cdot k\cdot(\lexp{g}{\lift{\psi}})^{-1}            \cdot \lexp{g}{\lift{\phi}}}{\iota(\lexp{h}{\lift{\phi}}\cdot\lexp{g}{\lift{\phi}})} \\
    &\qquad\quad \cdot\beta(h,k)\cdot \alpha(g,h) \cdot
      \alpha(g,k)\\
    &\qquad=g\cdot \lexp{g}{\lift{\psi}}\cdot h\cdot(\lexp{g}{\lift{\psi}})^{-1}\cdot \lexp{h}{\lift{\phi}}\cdot
     \lexp{g}{\lift{\psi}}\cdot k\cdot(\lexp{g}{\lift{\psi}})^{-1}  \cdot(\lexp{h}{\lift{\phi}})^{-1}\\
    &\qquad\quad \cdot  \beta(h,k)\cdot \alpha(g,h) \cdot
      \alpha(g,k)\\
     &\qquad=g\cdot \lexp{g}{\lift{\psi}}\cdot h\cdot(\lexp{g}{\lift{\psi}})^{-1}\cdot \lexp{k}{\iota(\lexp{h}{\lift{\phi}}\cdot \lexp{g}{\lift{\psi}})} \cdot \beta(h,k)\cdot \alpha(g,h) \cdot
      \alpha(g,k)\\
    &\qquad=g\cdot \lexp{g}{\lift{\psi}}\cdot h\cdot(\lexp{g}{\lift{\psi}})^{-1}\cdot \lexp{k}{\iota(\lexp{g}{\lift{\psi}}\cdot\lexp{h}{\lift{\phi}})} \cdot \beta(h,k)\cdot \alpha(g,h) \cdot
      \alpha(g,k)\\
    &\qquad=g\cdot \lexp{g}{\lift{\psi}}\cdot h\cdot\lexp{h}{\lift{\phi}}\cdot k\cdot (\lexp{h}{\lift{\phi}})^{-1}\cdot  \beta(h, k)\cdot(\lexp{g}{\lift{\psi}})^{-1}\cdot \alpha(g,h)\cdot\alpha(g,k)\\
    &\qquad=g\circ_{\psi,\alpha}(h\cdot\lexp{h}{\lift{\phi}}\cdot k\cdot (\lexp{h}{\lift{\phi}})^{-1}\cdot  \beta(h, k)) \\
    &\qquad=g\circ_{\psi,\alpha}(h\circ_{\phi,\beta}k),
  \end{align*}
  as claimed. 
\end{proof}

\begin{remark}
  Specialising $\phi(xK)=K$  and $\beta(x,y)=1$ for all  $x,y\in G$ in
  Theorem~\ref{thm:maintheorem}, so that $g  \circ_{\phi, \beta} h = g
  \cdot h$, we obtain that $(G,\cdot,\circ_{\psi,\alpha})$ is
  a bi-skew brace.
\end{remark}

\section{Iterating the construction}
\label{sec:iterating}

Let  $(\wasoverline\psi,\alpha)\in \mathcal{A}\times\mathcal{B}$,  and
let $\lift{\psi}$  be a lifting of  $\wasoverline\psi$.  Write ``$\circ$''
for ``$\circ_{\psi,\alpha}$''.
Then $(G,\circ)$ is  a group and $(K,\circ)$ and  $(A,\circ)$ are both
subgroups of $(G,\circ)$. (Note in particular that ``$\circ$'' and ``$\cdot$''
coincide   on   $K$,   as   $(K,\cdot)\le   Z(G,\cdot)$.)    Moreover,
$\lexp{K}{\lift{\psi}}\le K$, as
\begin{equation*}
  (\lexp{k}{\lift{\psi}}) K
  =
  \lexpp{k K}{\wasoverline\psi}
  =
  \lexp{K}{\wasoverline\psi}=K.
\end{equation*}
We  obtain that  $K\le  Z(G,\circ)$, and  since $(A/K,\circ)$  remains
abelian, we can consider the ring
\begin{equation*}
  \mathcal{A}_{\circ}
  =
  \Set{
    \wasoverline{\phi}\in\End(G/K,\circ)
    :
    \lexp{G/K}{\wasoverline\phi}\le A/K
  }.
\end{equation*}
Note that the bilinear map $\alpha$ does  not play a role in the group
operation of $G/K$, as the image  of $\alpha$ is contained in $K$.  We
claim  that $\mathcal{A}\subseteq  \mathcal{A}_{\circ}$. We  prove
this result in a more general setting.

\begin{lemma}
  \label{lem_weiterate}
  Let $B$ be a group, and let $C \le B$ be an abelian subgroup.

  Consider the ring $R$ of endomorphisms  of $B$ with image in $C$. If
  $\psi, \phi  \in R$,  then $\phi$  is also  an endomorphism  of $(B,
  \circ_{\psi})$, where
  \begin{equation*}
    x \circ_{\psi} y
    =
    x\cdot \lexp{x}{{\psi}}\cdot y\cdot (\lexp{x}{{\psi}})^{-1}.
  \end{equation*}
\end{lemma}

\begin{proof}
  For all $x,y\in B$ we have, as $C$ is abelian,
  \begin{equation*}
    \lexpp{x\circ_{\psi}y}{\phi}
    =
    \lexpp{x\cdot \lexp{x}{{\psi}}\cdot y\cdot
      (\lexp{x}{{\psi}})^{-1}}{\phi}
    =
    \lexp{x}{\phi}\cdot \lexpp{\lexp{x}{{\psi}}}{\phi}\cdot
    \lexp{y}{\phi}\cdot \lexpp{\lexp{x}{{\psi}}}{\phi}^{-1}
    =
    \lexp{x}{\phi}\cdot \lexp{y}{\phi}
  \end{equation*}
  and
  \begin{equation*}
    \lexp{x}{\phi}\circ_{\psi}\lexp{y}{\phi}
    =
    \lexp{x}{\phi}\cdot \lexpp{\lexp{x}{\phi}}{{\psi}}\cdot \lexp{y}{\phi}\cdot \lexpp{\lexp{x}{\phi}}{{\psi}}^{-1}
    =
    \lexp{x}{\phi}\cdot \lexp{y}{\phi}.\qedhere
  \end{equation*}
\end{proof}

Now consider 
\begin{align*}
  \mathcal{B}_{\circ}
    =
    \{&\beta\colon G\times G\to K
    :
    \text{$\beta$ is bilinear with respect to $\circ$, and}\\
    & \beta(K,G) = \beta(G,K) = 1
    \}.
\end{align*}
\begin{lemma}\label{lem_weiterate2}
  The inclusion $\mathcal{B}\subseteq\mathcal{B}_{\circ}$ holds.
\end{lemma}

\begin{proof}
  Let $\beta\in \mathcal{B}$, and let $g,h,k\in G$. Then
  \begin{align*}
    \beta(g\circ h,k)
    &=
    \beta(g\cdot \lexp{g}{\lift{\psi}}\cdot h\cdot(\lexp{g}{\lift{\psi}})^{-1}\cdot
    \alpha(g,h),k)
    \\&=
    \beta(g,k)\cdot \beta(h,k)
    =
    \beta(g,k)\circ\beta(h,k).
  \end{align*}
  Similarly,
  \begin{equation*}
    \beta(g,h\circ k)=\beta(g,k)\circ\beta(h,k).\qedhere
  \end{equation*}
\end{proof}

This  result  implies, very  much  as  in~\cite{Koc22}, that  given
$(\wasoverline{\phi},\beta)\in\mathcal{A}\times\mathcal{B}\subseteq
\mathcal{A}_{\circ}\times \mathcal{B}_{\circ}$, if $\lift{\phi}$ is a lifting of
$\wasoverline\phi$, then  we may  iterate  our main  construction with
respect   to   $(G,\circ)$   and  $(\phi, \beta)$.   Therefore
$(G,\circ,\ast)$ is a bi-skew brace, where
\begin{equation*}
  g\ast h
  =
  g\circ\lexp{g}{\lift{\phi}}\circ
  h\circ\overline{\lexp{g}{\lift{\phi}}}\circ\beta(g,h). 
\end{equation*}
(Here an overline
denotes the inverse with respect to
$\circ$).

Two questions  arise naturally.  Is $(G,\cdot,\ast)$ a  bi-skew brace?
And  if  this  is the  case,  can  we  obtain  it directly  from  our
construction applied to $(G,\cdot)$?

The next result provides an  affirmative answer to both questions. 

\begin{theorem}
  \label{thm:main2}
  Set $\circ_0$ to be ``$\cdot$''. Let
  $((\wasoverline{\psi_n},\alpha_n) : n\ge 1)$ 
  be a sequence of elements of $\mathcal{A}\times\mathcal{B}$, and for
  all $n\ge 1$, define
  \begin{equation*}
    g\circ_n h
    =
    g\circ_{n-1}
    \lexp{g}{\lift{\psi}_n}\circ_{n-1}h \circ_{n-1}
    \overline{\lexp{g}{\lift{\psi}_n}}\circ_{n-1}\alpha_{n}(g,h),    
  \end{equation*} 
  where   $\lift{\psi}_n$    is   a    lifting   of
  $\wasoverline{\psi_n}$   and
  an overline
  denotes the  inverse with  respect to
  $\circ_{n-1}$. Then for all $n\ge 1$ we have
  \begin{equation*}
    g\circ_n h
    =
    g\cdot \lexp{g}{q_n}\cdot h\cdot (\lexp{g}{q_n})^{-1}\cdot \beta_n(g,h),
  \end{equation*} 
  where \begin{equation*}
    q_n
    =
    \sum_{\substack{A\subseteq\{1,\cdots, n\}\\ A\ne \emptyset}}
    \left(\prod_{j\in A} \lift{\psi}_j\right)=
    \lift{\left(\sum_{\substack{A\subseteq\{1,\cdots, n\}\\ A\ne \emptyset}}
    \left(\prod_{j\in A} \psi_j\right)\right)}
  \end{equation*}
  and $\beta_n\in\mathcal{B}$. In particular, $(G,\cdot,\circ_n)$ is a
  bi-skew brace.
\end{theorem}

\begin{remark}
  \label{rem:the-reason-why}
  The previous result is the main reason why we have introduced
  bilinear maps in our construction in Section~\ref{sec:main}. In
  the following proof one sees indeed that bilinear maps 
  naturally occur in the form of  products of suitable central
  commutators in the iterating steps, even if one starts the
  construction with an element of $\mathcal{A}$ only.
\end{remark}

Before the proof, we introduce a technical lemma, in which we compute
conjugates with respect to our new operations.  
\begin{lemma}\label{lemma: technical}
  Let $(\psi, \alpha)\in \mathcal{A} \times \mathcal{B}$. Then
  for all $x,y\in G$ we have  
  \begin{equation*}
    x\circ_{\psi,\alpha} y\circ_{\psi,\alpha} \overline{x}
    =
    x\cdot\lexp{x}{\lift{\psi}}\cdot y\cdot(\lexp{x}{\lift{\psi}})^{-1}\cdot \lexp{y}{\lift{\psi}}\cdot x^{-1}\cdot (\lexp{y}{\lift{\psi}})^{-1}\cdot \alpha(x,y)\cdot \alpha(y,x^{-1}),
  \end{equation*}
  where an overline
  denotes the  inverse with  respect to
  $\circ_{\psi,\alpha}$.
\end{lemma}

\begin{proof}
	Write ``$\circ$'' for ``$\circ_{\psi,\alpha}$''. First, we compute $y\circ \overline{x}$:
	\begin{equation*}
		y\circ \overline{x}=y\cdot \lexp{y}{\lift{\psi}}\cdot \lexpq{x}{\lift{\psi}}^{-1}\cdot x^{-1}\cdot \lexp{x}{\lift{\psi}}\cdot \alpha(x,x)\cdot \lexpq{y}{\lift{\psi}}^{-1}\cdot \alpha(y,x^{-1}),
	\end{equation*}
	Therefore, employing~\eqref{eq:iota}, 
we deduce that 
	\begin{align*}
		 x\circ y\circ \overline{x} &= x \cdot  \lexp{x}{\lift{\psi}}\cdot (y\cdot \lexp{y}{\lift{\psi}}\cdot \lexpq{x}{\lift{\psi}}^{-1}\cdot x^{-1}\cdot \lexp{x}{\lift{\psi}}\cdot \alpha(x,x)\cdot \lexpq{y}{\lift{\psi}}^{-1}\cdot \alpha(y,x^{-1}))\\
		 &\quad\cdot \lexpq{x}{\lift{\psi}}^{-1}\cdot \alpha(x,y)\cdot \alpha(x,x^{-1})\\
		 &=x \cdot  \lexp{x}{\lift{\psi}}\cdot y\cdot \lexpq{x}{\lift{\psi}}^{-1}\cdot \lexp{y}{\lift{\psi}}\cdot x^{-1}\cdot \lexpq{y}{\lift{\psi}}^{-1}\cdot \lexp{x}{\lift{\psi}}\cdot \lexpq{x}{\lift{\psi}}^{-1} \\
		 &\quad\cdot \alpha(x,y)\cdot \alpha(y,x^{-1})\\
		 &= x\cdot\lexp{x}{\lift{\psi}}\cdot y\cdot(\lexp{x}{\lift{\psi}})^{-1}\cdot \lexp{y}{\lift{\psi}}\cdot x^{-1}\cdot (\lexp{y}{\lift{\psi}})^{-1}\cdot \alpha(x,y)\cdot \alpha(y,x^{-1}).
	\end{align*}
	as claimed.
\end{proof}

\begin{proof}[Proof of Theorem~\ref{thm:main2}]
  We prove  the result  by induction  on $n$.  For $n=1$,  the result
  follows from Theorem~\ref{thm:maintheorem},   with  $q_1=\lift{\psi}_1$   and
  $\beta_1=\alpha_1$.
  
  Suppose that the assertion holds for $n-1$, where $n\ge 2$, so that
  we have
  \begin{equation*}
    g\circ_{n-1} h
    =
    g\cdot \lexp{g}{q_{n-1}}\cdot h\cdot (\lexp{g}{q_{n-1}})^{-1}\cdot
    \beta_{n-1}(g,h). 
  \end{equation*} 
  Thus   $(G,\cdot,\circ_{n-1})$  is   a  bi-skew   brace,  again   by
  Theorem~\ref{thm:maintheorem}.

  Write $\lift{\psi}$ for  $\lift{\psi}_{n}$, $\alpha$ for $\alpha_n$,
  $q$ for $q_{n-1}$, and $\beta$ for $\beta_{n-1}$. Then we have
  \begin{equation*}
    g\circ_{n} h
    =
    g\circ_{n-1}
    \lexp{g}{\lift{\psi}}
    \circ_{n-1} h \circ_{n-1}
    \overline{\lexp{g}{\lift{\psi}}}
    \circ_{n-1}
    \alpha(g,h), 
  \end{equation*}
  where $\overline{\lexp{g}{\lift{\psi}}}$ denotes the inverse with respect
  to $\circ_{n-1}$. 
  By Lemma~\ref{lemma: technical} we have
  \begin{align*}
    &\lexp{g}{\lift{\psi}}\circ_{n-1}h\circ_{n-1}\overline{\lexp{g}{\lift{\psi}}}\\
    &\qquad=\lexp{g}{\lift{\psi}}\cdot \lexp{g}{q\lift{\psi}}\cdot h\cdot (\lexp{g}{q\lift{\psi}})^{-1}\cdot \lexp{h}{q}\cdot (\lexp{g}{\lift{\psi}})^{-1}\cdot (\lexp{h}{q})^{-1}\cdot \beta(\lexp{g}{\lift{\psi}},h)\cdot \beta(h,(\lexp{g}{\lift{\psi}})^{-1})\\
    &\qquad=\lexp{g}{\lift{\psi}}\cdot \lexp{g}{q\lift{\psi}}\cdot h\cdot (\lexp{g}{q\lift{\psi}})^{-1}\cdot (\lexp{g}{\lift{\psi}})^{-1}\cdot [\lexp{g}{\lift{\psi}},\lexp{h}{q}]\cdot \beta(\lexp{g}{\lift{\psi}},h)\cdot \beta(h,(\lexp{g}{\lift{\psi}})^{-1})\\
    &\qquad=\lexp{g}{\lift{\psi}+q\lift{\psi}}\cdot h\cdot(\lexp{g}{\lift{\psi}+q\lift{\psi}})^{-1}\cdot [\lexp{g}{\lift{\psi}},\lexp{h}{q}]\cdot \beta(\lexp{g}{\lift{\psi}},h)\cdot \beta(h,(\lexp{g}{\lift{\psi}})^{-1}).
  \end{align*}
  We obtain that 
  \begin{align*}
    &g\circ_{n-1}(\lexp{g}{\lift{\psi}+q\lift{\psi}}\cdot h\cdot(\lexp{g}{\lift{\psi}+q\lift{\psi}})^{-1}\cdot [\lexp{g}{\lift{\psi}},\lexp{h}{q}]\cdot \beta(\lexp{g}{\lift{\psi}},h)\cdot \beta(h,(\lexp{g}{\lift{\psi}})^{-1}))\\
    &\qquad=g\cdot \lexp{g}{q}\cdot \lexp{g}{\lift{\psi}+q\lift{\psi}}\cdot h\cdot(\lexp{g}{\lift{\psi}+q\lift{\psi}})^{-1}\cdot [\lexp{g}{\lift{\psi}},\lexp{h}{q}]\cdot \beta(\lexp{g}{\lift{\psi}},h)\cdot \beta(h,(\lexp{g}{\lift{\psi}})^{-1})\\
    &\qquad\quad \cdot (\lexp{g}{q})^{-1}\cdot \beta(g,h)\\
    &\qquad= g\cdot \lexp{g}{q+\lift{\psi}+q\lift{\psi}}\cdot h\cdot (\lexp{g}{q+\lift{\psi}+q\lift{\psi}})^{-1}\\
    &\qquad \quad\cdot [\lexp{g}{\lift{\psi}},\lexp{h}{q}]\cdot \beta(\lexp{g}{\lift{\psi}},h)\cdot \beta(h,(\lexp{g}{\lift{\psi}})^{-1})
    \cdot \beta(g,h).
  \end{align*}
   (Here we have used the fact that $[\lexp{g}{\lift{\psi}},\lexp{h}{q}]\in K\le
  Z(G)$, as $A / K$ is abelian.)
  Finally, note that if $k\in K$ and $x\in G$, then we have $x\circ_{n-1}
  k=x\cdot k$; thus, as  $\alpha(g,h)\in K$, we find that $g\circ_n h$ equals
  \begin{align*}
    &g\cdot \lexp{g}{q+\lift{\psi}+q\lift{\psi}}\cdot h\cdot
    (\lexp{g}{q+\lift{\psi}+q\lift{\psi}})^{-1}\\
    &\qquad\quad\cdot
    [\lexp{g}{\lift{\psi}},\lexp{h}{q}]\cdot \beta(\lexp{g}{\lift{\psi}},h)\cdot
    \beta(h,(\lexp{g}{\lift{\psi}})^{-1}) 
    \cdot \beta(g,h)\cdot \alpha(g,h).
  \end{align*}
  Therefore we may set 
  \begin{equation*}
    \beta_{n}(g,h)
    =
    [\lexp{g}{\lift{\psi}},\lexp{h}{q}]\cdot \beta(\lexp{g}{\lift{\psi}},h)\cdot
    \beta(h,(\lexp{g}{\lift{\psi}})^{-1})\cdot \beta(g, h) \cdot \alpha(g,h), 
  \end{equation*}
  which can be easily seen to satisfy the required properties, and we note that 
  \begin{align*}
    &q+\lift{\psi}+q\lift{\psi}=q_{n-1}+\lift{\psi}_{n}+q_{n-1}\lift{\psi}_{n}\\
    &\qquad=\sum_{\substack{A\subseteq\{1,\cdots, n-1\}\\ A\ne
        \emptyset}} \left(\prod_{j\in A}
    \lift{\psi}_j\right)+\lift{\psi}_{n}+\sum_{\substack{A\subseteq\{1,\cdots,
        n-1\}\\ A\ne \emptyset}} \left(\prod_{j\in A}
    \lift{\psi}_j\right)\lift{\psi}_{n}\\ 
    &\qquad=\sum_{\substack{A\subseteq\{1,\cdots, n\}\\ A\ne
        \emptyset}} \left(\prod_{j\in A} \lift{\psi}_j\right)=q_{n}.
  \end{align*}
  We  conclude,  appealing once more  to Theorem~\ref{thm:maintheorem},  that
  $(G,\cdot,\circ_n)$ is a bi-skew brace.
\end{proof}

\begin{corollary}\label{cor:main2}
  In  the  setting of  Theorem~\ref{thm:main2},
  $(G,\circ_n,\circ_m)$ is a bi-skew brace,  for  all $n,m\ge  0$.
\end{corollary}
\begin{proof}
  Assuming without  loss of generality  that $n  \le m$, we  can apply
  Theorem~\ref{thm:main2}    starting     with    $(G,\circ_n)$,    as
  $( (\wasoverline{\psi}_{i+n},\alpha_{i+n}) : i\ge 1)$ is a sequence of
  elements of $\mathcal{A}_{\circ_n}\times \mathcal{B}_{\circ_n}$ by
  Lemmas~\ref{lem_weiterate} and~\ref{lem_weiterate2}.
\end{proof}

\section{Examples}
\label{sec:examples}

\subsection{Koch's construction revisited}

We begin by recovering Koch's construction of~\cite{Koc22}.

\begin{example}
  \label{exa:koch}
  Let ${\psi}\colon G\to G$ be an abelian map, that is, an endomorphism of
  $G$     such     that     $\lexp{G}{{\psi}}=A$    is     abelian.

  Take
  $K=\lexp{[G,G]}{{\psi}}=1\le  Z(G)$. Then  ${\psi}\in\mathcal{A}$, and  we
  can consider $\phi=-{\psi}\in \mathcal{A}$, where
  \begin{equation*}
    \lexp{g}{\phi}
    =
    \lexpp{g^{-1}}{{\psi}}.
  \end{equation*}	
  Now take  $( \phi_n : n\ge 1)$,  where $\phi_n=\phi$ for  all $n\ge
  1$.  Then Theorems~\ref{thm:maintheorem}~and \ref{thm:main2}
  yield a  family of  operations 
  $( \circ_n : n\ge 1 )$ defined by
  \begin{equation*}
    g\circ_n h
    =
    g\cdot \lexp{g}{q_n}\cdot h\cdot (\lexp{g}{q_n})^{-1}
  \end{equation*}
  such that $(G,\circ_n,\circ_m)$  is a bi-skew brace  for all $n,m\ge
  1$. Here
  \begin{align*}
    q_n
    =
    \sum_{\substack{A\subseteq\{1,\cdots,      n\}
        \\
        A\ne \emptyset}}
    \phi^{|A|}=\sum_{i=1}^n            \binom{n}{i}\phi^i=\sum_{i=1}^n
    \binom{n}{i}(-{\psi})^i.
  \end{align*}
  Note that all the central bilinear maps are trivial, as they have
  image in $K=1$. 
\end{example}

\begin{example}
  \label{ex:variation}
  In  a  slight variation  on  Example~\ref{exa:koch}, and  with the  same
  setting, we  may take  $( \psi_n : n\ge 1 )$,  where ${\psi}_n={\psi}$
  for  all   $n\ge  1$.  Then   we  obtain  a  family   of  operations
  $( \circ_n : n\ge  1 )$  such   that  $(G,\circ_n,\circ_m)$  is  a
  bi-skew brace for all $n,m\ge 1$, where
  \begin{equation*}
    g\circ_n h
    =
    g\cdot \lexp{g}{q_n}\cdot h\cdot (\lexp{g}{q_n})^{-1}
  \end{equation*}
  and 
  \begin{align*}
    q_n
    =
    \sum_{\substack{A\subseteq\{1,\cdots, n\}
        \\
        A\ne \emptyset}} {\psi}^{\Size{A}}
    =
    \sum_{i=1}^n \binom{n}{i} {\psi}^i.
  \end{align*}
\end{example}

We  can  also  give  the following  slight  generalisation  of  Koch's
construction, whose proof  follows as an immediate
  application of Theorem~\ref{thm:maintheorem}.
\begin{proposition}
  Let $A$ be an abelian subgroup of $G$, and let 
  \begin{equation*}
    \mathcal{A}=\{\psi\in \End(G)\mid \lexp{G}{\psi}\le A\}.
  \end{equation*}
  For all $\psi\in \mathcal{A}$, write 
  \begin{equation*}
    g\circ_{\psi} h=g\cdot \lexp{g}{\psi}\cdot h\cdot \lexp{g}{\psi}^{-1}.
  \end{equation*}
  Then $(\circ_{\psi}\mid \psi\in \mathcal{A})$ is a brace block on $G$. 
\end{proposition}

\subsection{Endomorphisms, liftings, and central bilinear maps}

The next example shows that the operations obtained with liftings are
not covered by endomorphisms and central bilinear maps alone.

If $G$ is  a group, then for  all $i\ge 0$ we  write $\gamma_i(G)$ for
the $i$-th  term of  the lower  central series,  and $Z_i(G)$  for the
$i$-th term of the upper central series.
\begin{example}
  Let $p \ge 5$  be a prime, and let $H$ the  group of upper unitriangular
  matrices in $\GL(5, p)$.

  Write as  usual $e_{ij}$ for a matrix whose
  $(i,  j)$-entry is  $1$ and  all other  entries are  zero, and  let
  $t_{ij} = 1 + e_{ij} \in H$, for $i < j$.
  
  The following elementary facts
  hold:
  \begin{enumerate}
  \item $H$ has exponent $p$.
  \item $\gamma_{k}(H) = \Span{ t_{ij} : j - i \ge k }$.
  \item $H$ has nilpotence class four.
  \item The upper and lower central series of $H$ coincide.
  \end{enumerate}

  Moreover we have
  \begin{equation*}
    [t_{12}, t_{23}] = t_{13},	\quad
    [t_{34}, t_{45}] = t_{35},\quad 
    [t_{13}, t_{35}] = t_{15} \ne 1.
  \end{equation*}
  Let $G = \Span{x, y} \times H$, where $\Span{x, y}$ is elementary
  abelian of order $p^{2}$. Let $K = \gamma_{4}(H) = \Span{t_{15}} \le
  G$, and let $A = \gamma_{2}(H) \le G$.

  We have the following facts:
  \begin{enumerate}
  \item $K \le Z(G)$.
  \item $A/K$ is abelian, while $A$ is not.
  \item $\Span{t_{13}, t_{35}} K / K$ is elementary abelian of order $p^{2}$.
  \item The assignment
    \begin{equation*}
      x \mapsto t_{13},\quad y \mapsto t_{35},\quad
      t_{ij} \mapsto 1
    \end{equation*}
    uniquely defines an endomorphism $\psi$ of $G/K$.
  \item The endomorphism $\psi$ does not extend to an endomorphism of
    $G$, because if $g \in t_{13} K$ and $h \in t_{35} K$, then $[g,
    h] = [t_{13}, t_{35}]$, as $K \le Z(G)$, so that  $[g,
    h] = t_{15} \ne 1$.
  \end{enumerate}

  We claim that there is no $\phi \in \End(G)$ and central bilinear map
  $\alpha \colon G \times G \to K$ such that, for all $g, h \in G$, one has
  \begin{equation}
    \label{eq:no-no}
    g \circ_{\psi, \trivial} h
    =
    g \cdot \lexp{g}{\phi} \cdot h \cdot \lexp{g}{\phi}^{-1} \cdot
    \alpha(g, h),
  \end{equation}
  where $\trivial \colon G \times G \to K$ is the trivial bilinear map,
  mapping everything to $1$.
  In other words, endomorphisms and central bilinear maps alone cannot
  replace liftings. 
  
  If~\eqref{eq:no-no} holds, set $g = x$ to get for all $h \in H$ that
  \begin{align*}
    x \circ_{\psi, \trivial} h
    &=
    x \cdot \lexp{x}{\lift{\psi}} \cdot h \cdot
    \lexp{x}{\lift{\psi}}^{-1}
    =
    x \cdot h \cdot [h^{-1}, t_{13}]\\
    &=
    x \cdot h \cdot [h^{-1}, \lexp{x}{\phi}] \cdot \alpha(x, h).
  \end{align*}
  Thus for all $h \in G$ we have
  \begin{align*}
    [h^{-1}, \lexp{x}{\phi}^{-1} \cdot t_{13}]
    &=
    [h^{-1}, \lexp{x}{\phi}^{-1}]\cdot [\lexp{x}{\phi}^{-1},[h^{-1},
        t_{13}]]\cdot [h^{-1}, t_{13}] 
    \\&\equiv
     [h^{-1}, \lexp{x}{\phi}^{-1}]\cdot  [h^{-1}, t_{13}]
      \equiv 1
      \pmod{K}.
  \end{align*}
  Since $K = \gamma_{4}(H) \le Z(G)$, we obtain that
  $\lexp{x}{\phi}^{-1} \cdot t_{13} = a \in Z_{2}(G)$.

  A  similar argument  yields $\lexp{y}{\phi}^{-1}
    \cdot    t_{35}   =    b    \in   Z_{2}(G)$.     Now   we    have
  $[\gamma_{2}(G),   Z_{2}(G)]  =   1$.   Moreover
    $Z_{2}(G) =  \Span{x, y} \times \gamma_{3}(H)$  is abelian.  Since
    $t_{13}, t_{35} \in \gamma_{2}(G)$, we obtain
  \begin{align*}
    1
    =
    \lexp{[x, y]}{\phi}
    =
    [t_{13} \cdot a^{-1}, t_{35} \cdot b^{-1}]
    =
    [t_{13}, t_{35}]
    =
    t_{35}
    \ne
    1,
  \end{align*}
  a contradiction.
\end{example}

We now sketch a more complicated construction, which cannot be
expressed as a direct product as the previous one.
\begin{example}
  Let
  \begin{equation*}
    G
    =
    \Span{
      x, y, t_{1}, t_{2}, t_{3}, t_{4}
      :
      \text{$[y, x] = 1$, and commutators of weight $> 4$ vanish}
      }.
  \end{equation*}
  So $G$ is the quotient group of the free group in the variety of
  nilpotent groups of class four in the generators $x, y, t_{1}, t_{2},
  t_{3}, t_{4}$, by the normal subgroup generated by $[y, x]$.

  Let $A = \gamma_{2}(G)$, and let $K = \gamma_{4}(G)$. Note that $A$ is
  non-abelian, as $[[t_{1}, t_{2}], [t_{3}, t_{4}]] \ne 1$, while $A /
  K$ is abelian, as $[\gamma_{2}(G), \gamma_{2}(G)] \le \gamma_{4}(G)$.
  
  The assignment \begin{equation*}
    x \mapsto  [t_{1}, t_{2}],
    \quad
    y \mapsto [t_{3}, t_{4}],
    \quad
    t_{i} \mapsto 1
  \end{equation*}
  extends to an endomorphism $\psi$ of $G/K$, as the
  quotient of $G$ by the normal subgroup generated by the $t_{i}$
  (which contains $K$) is
  free abelian of rank two, and so is
  $\Span{[t_{1}, t_{2}], [t_{3}, t_{4}]} K / K$. However, very much as
  in the previous example, $\psi$ one sees does
  not extend to an endomorphism of $G$

  One  can check  with the  Nilpotent Quotient  Algorithm~\cite{nq} of
  \textsf{GAP}~\cite{GAP4} that $Z(G)  = \gamma_{4}(G)$, and $Z_{2}(G)
  = \gamma_{3}(G)$. Then an argument entirely similar to the one in
  the previous example shows that an identity like~\eqref{eq:no-no}
  does not hold here.
\end{example}

\subsection{Groups of nilpotence class two}

Let $G$ be a group of nilpotence class two. Our construction applies
to $G$ with $A = G$ and $K = [G,G]$. Here $\mathcal{A}
= \End(A/K)$. In particular, we obtain the following two classes of examples.

\begin{proposition}\label{pro:npower}
  Let $G$ be a group of nilpotence class two. For all $n\in N$ write 
  \begin{equation*}
    g \circ_{n} h
    =
    g \cdot g^{n} \cdot h \cdot g^{-n}
    =
    g \cdot h \cdot [h^{-1}, g]^{n}
    =
    g \cdot h \cdot [g, h]^{n}.
  \end{equation*}
  Then $(\circ_n\mid n\in \Z)$ is a brace block on $G$.
\end{proposition}
\begin{proof}
  As $G/K$ is abelian, $gK\mapsto g^nK$ is an endomorphism of $G/K$. Therefore we can apply 
  Theorem~\ref{thm:maintheorem} to derive the assertion. 
\end{proof}

\begin{example}
  \label{ex:car-auto}
  In~\cite{car-auto}, the first author  constructed examples of finite
  $p$-groups of class  two, for $p > 2$, where  $[G,G]$ coincides with
  the Frattini  subgroup $\Frat(G)$, such that  $\End(G)$ induces only
  the  identity   map  $\mathbf{1}$   and  the   trivial  endomorphism
  $\mathbf{0}$ on  $G/[G,G]$ (there are  many similar examples  in the
  literature).
  
  In particular, if $n\not\equiv 0,1\pmod p$, then $gK\mapsto g^nK$ is
  an endomorphism  of the elementary abelian  group $p$-group $G/Z(G)$
  which cannot  be lifted to an  endomorphism of $G$. This  means that
  the operation $\circ_n$ defined by
  \begin{equation*}
    g \circ_{n} h = g \cdot g^{n} \cdot h \cdot g^{-n}
    =
    g \cdot h \cdot [g, h]^{n}
  \end{equation*}
  can be  obtained either with a  lifting or with a  bilinear map, but
  not with an endomorphism alone.
\end{example}

\begin{remark}
  See~\cite{CS22-p} for an application of Proposition~\ref{pro:npower}
  to skew braces and Rota--Baxter operators
  (\cite{RB}).
\end{remark}

\begin{proposition}
  Let  $G$  be  a  group   of  nilpotence  class  two.  For  all
  $\psi\in\End(G)$ write
  \begin{equation*}
    g\circ_{\psi}   h=g\cdot  \lexp{g}{\psi}\cdot   h\cdot
    \lexp{g}{\psi}^{-1}.
  \end{equation*}
  Then  $(\circ_{\psi}: \psi\in  \End(G))$ is  a brace  block on
  $G$.
\end{proposition}
\begin{proof}
  As $K=[G,G]$ is fully invariant in $G$, every endomorphism of $G$ is
  a lifting of  the endomorphism it induces on $G/K$.   We conclude by
  applying Theorem~\ref{thm:maintheorem}.
\end{proof}

The following  example allows us  to obtain  brace block of  any given
cardinality, consisting of distinct  operations, because
there are  commutative rings $S$  with unity  of any  given
cardinality. (The latter statement appears to be folklore, see for
instance~\cite{card}: if the cardinality $n$ is finite, one may take
$S = \Z / n \Z$. If the cardinality $n$ is infinite, one may take $S$
to be the
polynomial ring in $n$ indeterminates with integer coefficients.)

\begin{example}
\label{ex:CG}
	Let $R$ be  a commutative ring with unity, and  let $G$ be the
        Heisenberg group
  \begin{equation*}
    G = \Set{
      (a,b,c)\mid a,b,c\in R
    },
  \end{equation*}
  with group operation
  \begin{equation*}
    (a,b,c)\cdot (a',b',c')=(a+a',b+b',c+c'+ab').
  \end{equation*}
  Then $G$ is a  group of nilpotence class
  two: an immediate computation shows that
  \begin{equation*}
    [(a,b,c),(a',b',c')]
    =
    [(0,0,ab'-a'b)]
    \in
    \Set{
      (0,0,c)\mid c\in R
    }
    =
    Z(G).
  \end{equation*}
  For all $x\in R$, consider the map
  \begin{equation*}
    \psi_x\colon (a,b,c)\mapsto (xa,xb,x^2c).
  \end{equation*}
  It  is clear  that  $\psi_x\in \End(G)$,  so  $(\psi_x\mid x\in  R)$
  yields a brace  block $(\circ_x\mid x\in R)$ on  $G$. Explicitly, if
  $g=(a,b,c)$ and $h=(a',b',c')$, then
  \begin{equation*}
    g\circ_{x} h = g\cdot h\cdot (0,0,x(ab'-a'b)).
  \end{equation*}
  We deduce that this brace block consists of all distinct operations,
  and has cardinality $\Size{R}$. We consider two special cases.
  \begin{itemize}
  \item Suppose  that $R$ has characteristic  zero.  Then $\Z\subseteq
    R$, and it  is immediate to see  that the skew braces  of the kind
    $(G,\cdot,\circ_n)$, $n\in \Z$, are  not isomorphic. We have found
    a   brace  block   consisting   of  (at   least)  countably   many
    non-isomorphic skew braces.
  \item  Suppose that  $R$  is  a topological  ring.   Then  $G$ is  a
    topological  group with  the product  topology, and  if we  take a
    sequence $(r_n\mid n\in \N)$ of elements of $R$ converging to $0$,
    we deduce that the corresponding operations $(\circ_{r_n}\mid n\in
    \N)$ are distinct  and converge to the original one,  in the sense
    that for all $g,h\in G$,
    \begin{equation*}
      g \circ_{\infty} h = \lim_{n\to \infty} g \circ_n h = g \cdot h.
  \end{equation*}
  \end{itemize} 
  For example, if  $R=\Z_p$, the ring of $p$-adic  integers, then both
  the previous cases apply; for the latter, one can take $r_n=p^n$ for
  all $n\ge  0$. This addresses  a question  posed to us  by Cornelius
  Greither at the Conference on Hopf algebras and Galois module theory
  held (online)  in Omaha in  May 2021,  about finding a  brace blocks
  with distinct operations converging to the original one.
\end{example}

\subsection{More brace blocks of any given cardinality}

\begin{example}
  \label{ex:arb_card-two}
  Let $(X, \le)$ be a well-ordered set with first element
  $x_{0}$. Let $(G, \cdot)$ be the free group of class two on $X$, and
  take $A = G$ and $K = G' = Z(G)$; the latter subgroup is free
  abelian in the commutators $[u, v]$, 
  for $u > v$.

  For all $y \in X$, consider the endomorphism $\psi_{y}$ of $G / K$
  defined by
  \begin{equation*}
    \begin{cases}
      \lexpp{x_{0} K}{\psi_{y}} = y K,\\
      \lexpp{x K}{\psi_{y}} = K, & \text{for $x \in X \setminus \Set{x_{0}}$.}\\
    \end{cases}
  \end{equation*}
  Then for $y \in X$ we have
  \begin{align*}
    x_{0} \circ_{\psi_{y}, \trivial} x_{0}
    &=
    x_{0} \cdot
    \lexpp{x_{0}}{\lift{\psi_{y}}} \cdot x_{0} \cdot
    \lexpp{x_{0}}{\lift{\psi_{y}}}^{-1} 
    \\&=
    x_{0}^{2} \cdot x_{0}^{-1} \cdot
    y \cdot x_{0} \cdot y^{-1}
    \\&=
    x_{0}^{2} \cdot [x_{0}^{-1}, y]
    =
    x_{0}^{2} \cdot [y, x_{0}],
  \end{align*}
  so that all operations $\circ_{\psi_{y}, \trivial}$ are distinct, for $y
  \in X$.
\end{example}

Another example of this kind can be constructed in the realm of
(absolutely) free groups.

\begin{example}
  \label{ex:arb_card-free}
  Let $X$ be a non-empty set, and $a \notin X$. Let $(G, \cdot)$ be the free
  group on $X \cup \Set{a}$, and let $K = 1$, $A = \Span{a}$.

  For all $y \in X$, consider the endomorphism $\psi_{y}$ of $G$
  defined by
  \begin{equation*}
    \begin{cases}
      \lexpp{y}{\psi_{y}} = a,\\
      \lexpp{a}{\psi_{y}} = 1,\\
      \lexpp{x}{\psi_{y}} = 1, & \text{for $x \in X \setminus \Set{y}$.}\\
    \end{cases}
  \end{equation*}
  Then we have $x \circ_{\psi_{y}, \trivial} x = x \cdot x$ for $x \ne y$,
  and $y \circ_{\psi_{y}, \trivial} y = y \cdot a \cdot y \cdot a^{-1} = y
  \cdot y \cdot [y^{-1}, a] \ne y
  \cdot y$, so that all operations $\circ_{\psi_{y}, \trivial}$ are distinct, for $y
  \in X$.
\end{example}

\section{Set-theoretic solutions of the Yang--Baxter equation}
\label{sec:YB}

We  recall that  a  \emph{set-theoretic solution  of the  Yang--Baxter
  equation},  defined  in~\cite{Dri92},  is   a  pair  $(X,r)$,  where
$X\ne\emptyset$ is a set and
\begin{align*}
  r\colon     X\times     X&\to      X\times     X\\
                      (x,y)&\mapsto  (\sigma_x(y),\tau_y(x))
 \end{align*}
is a bijective map satisfying
\begin{equation*}
  (r\times\id_X)(\id_X\times             r)(r\times\id_X)=(\id_X\times
  r)(r\times\id_X)(\id_X\times r).\end{equation*}
We say that $(X,r)$ is
\begin{itemize}
\item \emph{non-degenerate} if for  all $x\in  X$, $\sigma_x$ and $\tau_x$  are
  bijective;
\item \emph{involutive} if $r^2=\id_{X\times X}$. 
\end{itemize}
In what follows, by a  \emph{solution} $(X,r)$ we mean a set-theoretic
non-degenerate solution of the Yang--Baxter equation.

The connection  between skew  braces and  set-theoretic non-degenerate
solutions of the  Yang--Baxter equation has been  developed in various
papers  in  the  last   few  years;  see,  for  example,~\cite{braces}
and~\cite{skew}.

In   the  next   statement,   we  summarise   some   of  the   results
of~\cite{braces,skew,Rum19,KT20a};   see~\cite[Section   5]{CS1}   for
further details.
\begin{proposition}
  \label{pro:solutions}
  Let  $(G,\cdot,\circ)$  be  a  skew  brace.  Write  $g^{-1}$  and
  $\overline{g}$ for  the inverse of  $g$ with respect to  $\cdot$ and
  $\circ$, respectively. Then $(G,r)$ and $(G,r')$ are solutions, where
  \begin{align*}
    r(g,h)&=(g^{-1}\cdot  (g\circ h),\overline{g^{-1}\cdot
      (g\circ h)} \circ  g \circ h),\\ r'(g,h)&=((g\circ
    h)\cdot  g^{-1},\overline{(g\circ h)\cdot  g^{-1}}  \circ g  \circ
    h).
  \end{align*}
  
  The solutions $(G, r_{\mathcal{G}})$ and $(G, r_{\mathcal{G}'})$ are
  one the inverse  of the other and coincide if  and only if $(G,\cdot
  )$ is abelian.
\end{proposition}
\begin{remark}
	In Proposition~\ref{pro:solutions}, if $(G,\cdot,\circ)$ is a bi-skew brace, then we can reverse the role of the two operations, obtaining other two solutions. 
\end{remark}

For the next result, assume the setting of Section~\ref{sec:main}.
\begin{theorem}\label{thm:yb}
  Let         $(\psi,\alpha),(\phi,\beta)\in         \mathcal{A}\times
  \mathcal{B}$. Then $(G,r)$ and $(G,r')$ are solutions, where
  \begin{align*}
    &r(g,h)=(\lexp{g}{\lift{\psi}-\lift{\phi}}
    \cdot               h \cdot
    (\lexp{g}{\lift{\psi}-\lift{\phi}})^{-1}\cdot
    \beta(g^{-1},h)\cdot\alpha(g,h),\\
    &\qquad
    (\lexp{h}{\lift{\psi}})^{-1}\cdot\lexp{g}{\lift{\psi}-\lift{\phi}}
    \cdot     h^{-1}\cdot
    (\lexp{g}{\lift{\psi}-\lift{\phi}})^{-1}\cdot g
    \cdot  \lexp{g}{\lift{\psi}}\cdot h\cdot
    (\lexp{g}{\lift{\psi}})^{-1}\cdot\lexp{h}{\lift{\psi}}\\
    &\qquad\quad
    \cdot \beta(g,h)\cdot\alpha(h^{-1},g)),\\
    &r'(g,h)=(g\cdot      \lexp{g}{\lift{\psi}}\cdot     h\cdot
    (\lexp{g}{\lift{\psi}})^{-1}\cdot     \lexp{h}{\lift{\phi}}\cdot
    g^{-1}\cdot
    (\lexp{h}{\lift{\phi}})^{-1}\cdot    \beta(h,g^{-1})\cdot    \alpha(g,h),
    \\
    &\qquad         \lexp{h}{\lift{\phi}-\lift{\psi}}\cdot         g\cdot
    (\lexp{h}{\lift{\phi}-\lift{\psi}})^{-1}
    \cdot \beta(h,g)\cdot\alpha(h^{-1},g)).\\
  \end{align*}
\end{theorem}

\begin{proof}
  By           Theorem~\ref{thm:maintheorem},
  $(G,\circ_{\phi,\beta},\circ_{\psi,\alpha})$ is a  bi-skew brace, so
  that we  can apply  Proposition~\ref{pro:solutions}.  The  (long but
  straightforward)  computation  to  exhibit  the  solutions  in  this
  explicit formulation is similar to that of~\cite[Theorem 5.3]{CS1}.
\end{proof}

\begin{corollary}\label{cor:solutions}
  Let  $(\psi,\alpha)\in \mathcal{A}\times  \mathcal{B}$.  Then $(G,r)$, $(G,r')$, $(G,\widetilde r)$, $(G,\widetilde{r}')$ are solutions, where
  \begin{align*}
    &r(g,h)=(\lexp{g}{\lift{\psi}}\cdot                    h\cdot
    (\lexp{g}{\lift{\psi}})^{-1}\cdot          \alpha(g,h),\\          &\qquad
    (\lexp{h}{\lift{\psi}})^{-1}\cdot\lexp{g}{\lift{\psi}}\cdot          h^{-1}\cdot
    (\lexp{g}{\lift{\psi}})^{-1}\cdot   g\cdot   \lexp{g}{\lift{\psi}}\cdot   h\cdot
    (\lexp{g}{\lift{\psi}})^{-1}\cdot\lexp{h}{\lift{\psi}}\cdot\alpha(h^{-1},g));\\
    &r'(g,h)=(g\cdot      \lexp{g}{\lift{\psi}}\cdot     h\cdot
    (\lexp{g}{\lift{\psi}})^{-1}\cdot  g^{-1}\cdot   \alpha(g,h),  \\  &\qquad
    (\lexp{h}{\lift{\psi}})^{-1}\cdot                                  g\cdot
    \lexp{h}{\lift{\psi}}\cdot\alpha(h^{-1},g));\\
    &\widetilde r(g,h)=((\lexp{g}{\lift{\psi}})^{-1}\cdot         h\cdot
    \lexp{g}{\lift{\psi}}\cdot           \alpha(g^{-1},h),\\           &\qquad
    (\lexp{g}{\lift{\psi}})^{-1}\cdot h^{-1}\cdot  \lexp{g}{\lift{\psi}}\cdot g\cdot
    h \cdot \alpha(g,h));\\
    &\widetilde r'(g,h)=(g\cdot      h\cdot     \lexp{h}{\lift{\psi}}\cdot
    g^{-1}\cdot (\lexp{h}{\lift{\psi}})^{-1}\cdot \alpha(h,g^{-1}), \\ &\qquad
    \lexp{h}{\lift{\psi}}\cdot g\cdot (\lexp{h}{\lift{\psi}})^{-1}\cdot \alpha(h,g)).
  \end{align*}
\end{corollary}

\begin{proof}
  It  is   enough to apply  Theorem~\ref{thm:yb} to the  bi-skew brace
  $(G,\cdot,\circ_{\psi,\alpha})$, as when both $\phi$ and $\beta$ are
  trivial,   the   operation   $\circ_{\phi,\beta}$   coincides   with
  ``$\cdot$''.
\end{proof}
\begin{remark}
  Note   that  the   solutions  given   in~\cite[Theorem  5.1]{KST20},
  \cite[Corollary    5.4]{Koch-Abelian}, and~\cite[Theorem   4.15]{Koc22}
   are particular  instances of the solutions
  exhibited in Theorem~\ref{thm:yb} and Corollary~\ref{cor:solutions}.
\end{remark}

\section{Brace blocks and graphs}

\label{sec:graphs}
We briefly discuss here the concept of a brace block from the point of
view of graphs.

Let $(G,1)$  be a pointed set.  Denote by $\Perm(G)$  the group of
permutations on the  set $G$. For $\eta\in \Perm(G)$  and $g\in G$,
write $\lexp{g}{\eta}$ for the image of $g$ under $\eta$.

\begin{definition}
  A subgroup $N$ of $\Perm(G)$ is \emph{regular} if the map
  \begin{align*}
    N&\to G\\ \eta&\mapsto \lexp{1}{\eta}
  \end{align*}
  is bijective.  We  denote by $\nu$ the inverse of  this map, so that
  $\nu(g)$    is   the    unique    element   of    $N$   such    that
  $\lexp{1}{\nu(g)}=g$, and $N=\Set{\nu(g):g\in G}$.
\end{definition}

\begin{example}[Cayley's theorem]
  \label{exa:lrr}
  Let ``$\circ$'' be  an operation on $G$  such that $(G,\circ)$
  is   a   group   with  identity   $1$.   Take   \begin{align*}
    \lambda_{\circ} \colon (G, \cdot) &\to \Perm(G)\\ g &\mapsto
    (h \mapsto g \circ h),
  \end{align*}
  the   \emph{left    regular   representation}   with    respect   to
  ``$\circ$''.  Then $\lambda_{\circ}(G)$  is  a  regular subgroup  of
  $\Perm(G)$.
\end{example}

Conversely,  every  regular subgroup  $N$  of  $\Perm(G)$ occurs  as
$\lambda_{\circ}(G)$, for a suitable  group operation ``$\circ$'' on
$G$.   Indeed, if  $N=\{\nu(g):g\in G\}$,  we can  use transport  of
structure  via  the bijection  $\nu$  to  obtain a  group  operation
``$\circ$'' on $G$  such that $\nu(g \circ h) =  \nu(g) \nu(h)$. It
is now immediate to see that  $\nu = \lambda_{\circ}$, so that $N$ is
precisely the image of the  left regular representation with respect
to ``$\circ$''. We have obtained the following result.
\begin{proposition}\label{pro:}
  Let  $(G,1)$  be  a  pointed   set.  The  following  data  are
  equivalent:
  \begin{enumerate}
  \item an  operation $\circ$ on $G$  such that $(G, \circ)$  is group
    with identity $1$;
  \item a regular subgroup $N=\Set{\nu(g):g\in G}\le\Perm(G)$.
  \end{enumerate}
  The correspondence is given by
  \begin{equation*}
    g\circ h=\lexp{h}{\nu(g)}.
  \end{equation*}
  In particular, the map
  \begin{equation*}
    \nu\colon (G,\circ)\to N
  \end{equation*}
  is an isomorphism, and $N=\lambda_{\circ}(G)$.
\end{proposition}

The  following result  is basically~\cite[Theorem  4.2]{skew}, with  a
slight change of point of view.

\begin{theorem}\label{thm:regskew}
  Let $(G,\circ_{1})$ and $(G,\circ_{2})$ be groups with identity $1$. Then
  $(G,\circ_{1},\circ_{2})$   is   a   skew    brace   if   and   only   if
  $\lambda_{\circ_{2}}(G)$ normalises $\lambda_{\circ_{1}}(G)$.
\end{theorem}

Now we can consider graphs.

\begin{definition}
  Let $(G,1)$ be a pointed set.  The \emph{normalising graph} of
  $G$  is the  undirected  graph $\kN$  whose  vertices are  the
  regular subgroups  of $\Perm(G)$,  and where two  vertices are
  joined by an  edge if and only if  the corresponding subgroups
  normalise each other.
\end{definition}

By Theorem~\ref{thm:regskew}, we  derive that having an  edge in the
normalising graph is equivalent to  having a bi-skew brace structure
on $G$.   In particular, \emph{cliques} (complete  subgraphs) in the
normalising graph of $(G, 1)$ correspond to brace blocks.

Similar  ideas and  correspondences  were  studied in~\cite{Koh21}.  A
particular case of the normalising graph is tackled in~\cite{Spa22}.

\providecommand{\bysame}{\leavevmode\hbox to3em{\hrulefill}\thinspace}
\providecommand{\MR}{\relax\ifhmode\unskip\space\fi MR }
\providecommand{\MRhref}[2]{%
  \href{http://www.ams.org/mathscinet-getitem?mr=#1}{#2}
}
\providecommand{\href}[2]{#2}

\end{document}